\documentclass[12pt]{amsart}
\usepackage{amsfonts,amssymb}

\newtheorem{theorem}{Theorem}

\newtheorem{lemma}{Lemma}

\newtheorem{proposition}{Proposition}
\newtheorem{corollary}{Corollary}

\theoremstyle{definition}
\newtheorem{definition}{Definition}
\theoremstyle{remark}

\newtheorem{remark}{Remark}

\newcommand{\mix}{\mathcal{M}}

\newcommand{\vv}{\mathcal{V}}
\newcommand{\ms}{\operatorname{MS}}
\newcommand{\fin}{{\operatorname{fin}}}

\newcommand{\zz}{\mathbb{Z}}

\newcommand{\cc}{\mathbb{C}}
\newcommand{\aaa}{\mathcal{A}}

\newcommand{\qq}{\mathbb{Q}}

\newcommand{\triv}{{\bf 1}}
\newcommand{\Ind}{\operatorname{Ind}}
\newcommand{\abs}[1]{\left|{#1}\right|}
\newcommand{\rep}[1]{\langle{#1}\rangle}

\newcommand{\bc}{\operatorname{bc}}
\newcommand{\rec}{\operatorname{rec}}
\newcommand{\res}{\operatorname{res}}

\newcommand{\pp}{\mathcal{P}}
\newcommand{\oo}{\mathcal{O}}
\newcommand{\weil}{\mathcal{G}}

\newcommand{\sss}{\mathcal{S}}

\newcommand{\Hom}{\operatorname{Hom}}

\newcommand{\odd}{\operatorname{odd}}

\numberwithin{theorem}{section} \numberwithin{lemma}{section}
\numberwithin{corollary}{section}
\numberwithin{proposition}{section}

\numberwithin{equation}{section}

\author{Omer Offen and Eitan Sayag}
\title{The $SL(2)$-type and base change}
\date\today

\begin{document}\maketitle
\begin{abstract}
The $SL(2)$-type of any smooth, irreducible and unitarizable
representation of $GL_n$ over a $p$-adic field was defined by
Venkatesh. We provide a natural way to extend the definition to
all smooth and irreducible representations. For unitarizable
representations we show that the $SL(2)$-type of a representation
is preserved under base change with respect to any finite
extension. The Klyachko model of a smooth, irreducible and
unitarizable representation $\pi$ of $GL_n$ depends only on the
$SL(2)$-type of $\pi.$ As a consequence we observe that the
Klyachko model of $\pi$ and of its base-change are of the same
type.
\end{abstract}
\section{introduction}

Let $F$ be a finite extension of $\qq_{p}$. In \cite{MR2133760},
Venkatesh assigned a partition of $n,$ the \emph{$SL(2)$-type} of
$\pi,$ to any smooth, irreducible and unitarizable representation
$\pi$ of $GL_n(F).$ For a representation of Arthur type the $SL(2)$-type encodes the combinatorial data in the Arthur parameter. In general,
the $SL(2)$-type is
defined in terms of
Tadic's classification of the unitary dual.

The reciprocity map for $GL_n(F)$ is a bijection from the set of isomorphism classes of smooth irreducible representations of $GL_n(F)$ to the set of isomorphism classes of $n$-dimensional Weil-Deligne representations (cf. \cite{MR1876802} and \cite{MR1738446}). Applying the reciprocity map
we observe that there is a natural way to extend the
definition of the $SL(2)$-type to all smooth and irreducible
representations of $GL_n(F)$ (see Theorem \ref{thm: SL(2)-type} and Remark \ref{rmk: SL(2)-type}).
The reciprocity map also allows
the definition of base change with respect to any finite extension
$E$ of $F.$ It is a map $\bc_{E/F}$ from isomorphism classes of
smooth irreducible representation of $GL_n(F)$ to isomorphism classes of
smooth irreducible representation of $GL_n(E)$ that is the `mirror image'
of restriction with respect to $E/F$ of Weil-Deligne
representations.
The content of Theorem \ref{thm: main}, our main result,
is that for any smooth, irreducible and unitarizable representation $\pi$ of
$GL_n(F)$ the representations
%the representation $\bc_{E/F}(\pi)$ is unitarizable and
$\pi$ and $\bc(\pi)$ have the same $SL(2)$-type.

In \cite{MR2332593}, \cite{os}, \cite{OS3} we studied the Klyachko
models of smooth irreducible representations of $GL_{n}(F),$ that
is, distinction of a representation with respect to certain
subgroups that are a semi direct product of a unipotent and a
symplectic group. Our results are also described in terms of
Tadic's classification and depend, in fact, only on the
$SL(2)$-type of a representation. For example, a smooth,
irreducible and unitarizable representation $\pi$ of $GL_{2n}(F)$
is $Sp_{2n}(F)$-distinguished, i.e. it satisfies
$\Hom_{Sp_{2n}(F)}(\pi,\cc) \ne 0,$ if and only if the
$SL(2)$-type of $\pi$ consists entirely of even parts (and in this
case $\Hom_{Sp_{2n}(F)}(\pi,\cc)$ is one dimensional \cite[Theorem
2.4.2]{MR1078382}). For unitarizable representations, our results
on Klyatchko models are reinterpreted here in terms of the
$SL(2)$-type. As a consequence we show that Klyachko models are
preserved under base-change with respect to any finite extension.
In particular, we have
\begin{theorem}\label{thm: symplectic main}
Let $E/F$ be a finite extension of $p$-adic fields. A smooth, irreducible
and unitarizable representation $\pi$ of $GL_{2n}(F)$ is
$Sp_{2n}(F)$-distinguished if and only if $\bc_{E/F}(\pi)$ is
$Sp_{2n}(E)$-distinguished.
\end{theorem}

The rest of this note is organized as follows. After setting some
general notation in Section \ref{sec: notation}, in Section
\ref{sec: bc} we recall the definition of the reciprocity map.
In Section \ref{sec: SL(2)-type} we recall the definition of Venkatesh for the
$SL(2)$-type of a unitarizable representation and extend it to all smooth irreducible representations.
We recall (and reformulate in terms of the
$SL(2)$-type) our results on symplectic (and more generally on
Klyachko) models in Section \ref{sec: Klyachko}. Our main observation
Theorem \ref{thm: main} and its application to Klyachko
models Corollary \ref{cor: main} are stated in Section \ref{sec: statements} and proved in Section \ref{sec: proofs}. The main theorem says that
base change respects $SL(2)$-types and its corollary says that base change respects Klyachko types. Theorem \ref{thm: symplectic main} is a
special case where the Klyachko type is purely symplectic.

\section{Notation}\label{sec: notation}
Let $F$ be a finite extension of $\qq_p$ for some prime number $p$
and let $\abs{\,\cdot}_F:F^\times \to \cc^\times$ denote the
standard absolute value normalized so that the inverse of
uniformizers are mapped to the size of the residual field. Denote
by $W_F$ the Weil group of $F$ and by $I_F$ the inertia subgroup
of $W_F.$ We normalize the reciprocity map $T_F:W_F \to F^\times,$
given by local class field theory, so that geometric Frobenius
elements are mapped to uniformizers. The map $T_F$ defines an
isomorphism from the abelianization $W_F^{ab}$ of $W_F$ to
$F^\times$ (this is the inverse of the Artin map). Let
$\abs{\,\cdot}_{W_F}=\abs{\,\cdot}_{F}\circ T_F$ denote the
associated absolute value on $W_F.$
%Note that with our
%normalizations, for a finite extension $E/F$ we have
%\begin{equation}\label{eq: abs val}
 %  \abs{w}_{W_E}=\abs{w}_{W_F},\ w \in W_E.
%\end{equation}

Denote by $\triv_\Omega$ the characteristic function of a set
$\Omega.$ Let $\ms_\fin(\Omega)$ be the set of finite multisets of
elements in $\Omega,$ that is, the set of functions $f:\Omega \to
\zz_{\ge 0}$ of finite support. When convenient we will also denote
$f$ by $\{\omega_1,\dots,\omega_1,\omega_2,\dots,\omega_2,\dots\}$
where $\omega\in \Omega$ is repeated $f(\omega)$ times. Let
$\pp=\ms_\fin(\zz_{>0})$ be the set of partitions of positive
integers and let
\[
    \pp(n)=\{f \in \pp: \sum_{k=1}^\infty k\,f(k)=n\}
\]
denote the subset of partitions of $n.$ For $n,\,m \in \zz_{>0}$ let
$(n)_m=m\,\triv_n=\{n,\dots,n\}$ be the partition of $nm$ with `$m$
parts of size $n$'. Let $\odd:\pp\to \zz_{\ge 0}$ be defined by
\[
    \odd(f)=\sum_{k=0}^\infty f(2k+1),
\]
i.e. $\odd(f)$ is the number of odd parts of the partition $f.$

\section{Reciprocity and base-change for $GL_n(F)$}\label{sec: bc}

\subsection{Weil-Deligne representations}

An $n$-dimensional \emph{Weil-Deligne} representation is a pair
$((\rho,V),N)$ where $(\rho,V)$ is an $n$-dimensional representation
of $W_{F}$ that decomposes as a direct sum of irreducible
representations and $N:V \to V$ is a linear operator such that
\[
    \abs{w}_{W_F}\,N \circ  \rho(w)=\rho(w)\circ N,\ w \in
    W_F.
\]
The map $((\rho,V),N) \mapsto ([\rho],f),$ where $[\rho]$ denotes
the isomorphism class of the $n$-dimensional representation
$(\rho,V)$ of $W_F$ and $f \in \pp(n)$ is the partition of $n$
associated to the Jordan decomposition of $N,$ defines an injective
map on isomorphism classes of Weil-Deligne representations. Denote
its image by $\weil_{F}(n).$ In this way we identify the set
$\weil_F(n)$ with the set of isomorphism classes of $n$-dimensional
Weil-Deligne representations. Let $P_{F,n}: \weil_F(n) \to \pp(n)$
be the projection to the second coordinate. Let
$\weil_F=\cup_{n=1}^\infty \weil_F(n)$ be the set of isomorphism
classes of all finite dimensional Weil-Deligne representations and
let $P_F:\weil_F \to \pp$ be the map such that
${P_F}_{|\weil_F(n)}=P_{F,n}.$

\subsection{The local Langlands correspondence}

Let $\aaa_F(n)$ be the set of isomorphism classes of smooth and
irreducible representations of $GL_n(F)$ and set
$\aaa_F=\cup_{n=1}^\infty \aaa_F(n).$ For every $\pi \in \aaa_F$ we
denote by $\omega_\pi$ the central character of (any representation
in the isomorphism class of) $\pi.$ Fix a non trivial additive
character $\psi$ of $F.$ Due to Harris-Taylor \cite{MR1876802} and
independently to Henniart \cite{MR1738446} there exists a unique
sequence of bijections
\[
    \rec_{F,n}: \aaa_{F}(n) \to \weil_{F}(n)
\]
for all $n\ge 1$ satisfying the following properties:
\begin{eqnarray}
\label{eq: LCFT}& &\rec_{F}(\chi)=\chi \circ T_{F};\\& & \label{eq:
L-id}L(\pi_{1} \times \pi_{2},s)=L(\rec_{F}(\pi_{1})\otimes
\rec_{F}(\pi_{2}),s);\\& & \label{eq: epsilon-id}\epsilon(\pi_{1}
\times \pi_{2},s,\psi)=\epsilon(\rec_{F}(\pi_{1})\otimes
\rec_{F}(\pi_{2}),s,\psi);\\& & \label{eq: central char}\det \circ
\rec_{F}(\pi)=rec_{F}(\omega_\pi);\\& &\label{eq: dual-id}
\rec_{F}(\pi^\vee)=\rec_{F}(\pi)^\vee.
\end{eqnarray}
Here $\chi \in \aaa_F(1),\,\pi,\,\pi_1,\,\pi_2 \in \aaa_F,$
$\pi^\vee$ is the contragredient of $\pi,$ $\rec_{F}(\pi)^\vee$ is
the dual of $\rec_{F}(\pi)$ and $\rec_F:\aaa_F \to \weil_F$ is such
that ${\rec_F}_{|\aaa_F(n)}=\rec_{F,n}.$

\subsection{Expressing $\rec_F$ in terms of $\rec_F^\circ$}
Let $\aaa_F^\circ(n)\subseteq \aaa_F(n)$ be the subset of
isomorphism classes of supercuspidal representations and let
$\weil_F^\circ(n)\subseteq \weil_F(n)$ be the subset of isomorphism
classes $([\rho],f)$ such that $\rho$ is irreducible and
$f=\triv_n=\{n\}.$ The set $\weil_F^\circ(n)$ is identified with the
set of isomorphism classes of irreducible and $n$-dimensional
representations of $W_F.$ It follows from the work of Harris-Taylor
and independently of Henniart that there exists a unique sequence of
bijections
\[
    {\rec_{F,n}}_{|\aaa_F^\circ(n)}=\rec_{F,n}^\circ:\aaa_F(n)\to \weil_F^\circ(n)
\]
satisfying \eqref{eq: LCFT}, \eqref{eq: L-id}, \eqref{eq:
epsilon-id}, \eqref{eq: central char} and \eqref{eq: dual-id}. The
work of Zelevinsky \cite{MR584084} allows the extention of
$\rec_F^\circ$ to the map $\rec_F$ on $\aaa_F.$ This is also
explained in \cite{MR902829} and we now recall the construction of
$\rec_F$ in terms of $\rec^\circ_F.$

For $s \in \cc$ and every isomorphism class $\varpi=[\pi] \in
\aaa_F$ (resp. $\varrho=([\rho],f) \in \weil_F$) let $\varpi[s]=[\pi
\otimes \abs{\det}_F^s]$ (resp. $\varrho[s]=([\rho\otimes
\abs{\,\cdot}_{W_F}^s],f)$). A \emph{segment} in $\aaa_F^\circ$
(resp. $\weil_F^\circ$) is a set of the form
\[
    \Delta[\sigma,r]=\{\sigma[\frac{1-r}2],\sigma[\frac{3-r}2],\dots,\sigma[\frac{r-1}2]\}
\]
(resp.
\[
    \Delta[\rho,r]=\{\rho[\frac{1-r}2],\rho[\frac{3-r}2],\dots,\rho[\frac{r-1}2]\})
\]
for some $\sigma \in \aaa_F^\circ$ (resp. $\rho \in \weil_F^\circ$)
and $r\in \zz_{>0}.$ Let $\sss$ (resp. $\sss'$) denote the set of
all segments in $\aaa_F^\circ$ (resp. $\weil_F^\circ$) and let
$\oo=\ms_\fin(\sss)$ (resp. $\oo'=\ms_\fin(\sss')$). The bijection
$\rec^\circ_F:\aaa_F^\circ \to \weil_F^\circ$ defines a bijection
$\rec_F^\circ:\sss\to \sss'$ given by
$\rec_F^\circ(\Delta[\sigma,r])=\Delta[\rec_F^\circ(\sigma),r]$ and
a bijection $\rec_F^\circ:\oo \to \oo'$ given by
$\rec_F^\circ(a)(\rec_F^\circ(\Delta))=a(\Delta),\,\Delta\in \sss.$

In \cite[Section 6.5]{MR584084} Zelevinsky defines a bijection $a
\mapsto \rep{a}$ from $\oo$ to $\aaa_F.$ The Zelevinsky involution
is defined in \cite[Section 9.12]{MR584084} as an involution on the
Grothendick group associated with $\aaa_F.$ It is proved by Aubert
\cite{MR1285969}, \cite{MR1390967b} and independently by Procter
\cite{MR1625483} that the Zelevinsky involution restricts to a
bijection from $\aaa_F$ to itself that we denote by $\pi \mapsto
\pi^t.$ In \cite[Section 10.2]{MR584084} Zelevinsky  defines a
bijection $\tau:\oo' \to \weil_F$ as follows. For a segment
$\Delta[\rho,r]\in \sss'$ where $\rho \in \weil_F^\circ(t)$ let
\[
    \tau(\Delta[\rho,r])=(\oplus_{i=1}^r \,\rho, (r)_t)
\]
and for $a' \in \oo'$ set
\[
    \tau(a')=\oplus_{\Delta' \in \oo'} \tau(\Delta')
\]
where for $([\rho_1],f_1),\dots,([\rho_m],f_m) \in \weil_F$ the
direct sum is given by
\[
     ([\rho_1],f_1)\oplus\cdots \oplus ([\rho_m],f_m)=([
    \rho_1 \oplus\cdots \oplus \rho_m],f_1+\cdots+f_m).
\]
The reciprocity map $\rec_F$ is given by
\[
    \rec_F(\rep{a}^t)=\tau(\rec_F^\circ(a)),\ a \in \oo.
\]

\section{The $SL(2)$-type of a representation}\label{sec: SL(2)-type}

Denote by $\aaa^u_F(n)$ the subset of $\aaa_F(n)$ consisting of all
isomorphism classes of unitarizable representations and let
$\aaa_F^u=\cup_{n=1}^\infty \aaa_F(n).$
For $[\pi_1],\dots,[\pi_m] \in
\aaa_F$ we denote by $\pi_1 \times \cdots \times \pi_m$ the
representation parabolically induced from $\pi_1 \otimes \cdots
\otimes \pi_m$ and by $[\pi_1] \times \cdots \times [\pi_m]$ its
isomorphism class.

For $\sigma \in \aaa_F^\circ$ and integers $n,\,r >0$ let
\[
    \delta[\sigma,n]=\rep{\Delta[\sigma,n]}^t,
\]
\[
    a(\sigma,n,r)=\{\Delta[\sigma[\frac{1-r}{2}],n],\Delta[\sigma[\frac{3-r}{2}],n],
    \cdots,\Delta[\sigma,n](\frac{r-1}{2})\}\in \oo
\]
and
\[
    U(\delta[\sigma,n],r)=\rep{a(\sigma,n,r)}.
\]
Tadic's classification of the unitary dual of $GL_n(F)$ \cite{MR870688}
implies that if $\sigma \in \aaa_F^\circ
\cap \aaa_F^u$ then $U(\delta[\sigma,n],r) \in \aaa_F^u$ and that
for any $\pi\in \aaa_F^u$ there exist $\sigma_1,\dots,\sigma_m \in
\aaa_F^\circ$ and integers $n_1,\dots,n_m,\,r_1,\dots,r_m >0$ such
that
\begin{equation}\label{eq: general unitary}
    \pi=U(\delta[\sigma_1,n_1],r_1)\times \cdots \times
    U(\delta[\sigma_m,n_m],r_m).
\end{equation}
It further follows from \cite[Lemma 3.3]{MR1359141} that
\begin{equation}\label{eq: transpose of U}
    U(\delta[\sigma,n],r)^t=U(\delta[\sigma,r],n).
\end{equation}
The $SL(2)$ of a representation $\pi\in \aaa_F^u$ of the form \eqref{eq:
general unitary} is defined in \cite[Definition 1]{MR2133760}
to be the partition
\begin{equation}\label{eq: explicit SL(2)-type}
%    \vv(\pi)=
\{(r_1)_{n_1},\dots,(r_m)_{n_m}\}.
\end{equation}
%We now observe that there is a natural way to extend the definition of the $SL(2)$-type to $\aaa_F.$
\begin{theorem}\label{thm: SL(2)-type}
The $SL(2)$-type of a representation $\pi\in \aaa_F^u$ equals
$P_F(\rec_F(\pi^t)).$
\end{theorem}
\begin{remark}\label{rmk: SL(2)-type}
Theorem \ref{thm: SL(2)-type} allows us to define the $SL(2)$-type
of any $\pi \in \aaa_F$ by the formula $P_F(\rec_F(\pi^t)).$ Note
further that given a reciprocity map (local Langlands conjecture),
this provides a recipe to define the $SL(2)$-type of an
irreducible representation for any reductive group!
\end{remark}
\begin{proof}
Based on Tadic's classification of the unitary dual of
$GL_n(F),$ the proof of Theorem \ref{thm: SL(2)-type} is merely a
matter of following the definitions. For convenience, we provide the
proof. The key is in the following simple observations.
\begin{lemma}\label{lemma: rec of unitary}
Let $\pi\in \aaa_F^u$ be of the form \eqref{eq: general unitary}.
Then
\begin{equation}\label{eq: rec of unitary}
    \rec_F(\pi)=\oplus_{i=1}^m \oplus_{j=1}^{r_i}
    \tau(\Delta[\sigma_i[\frac{r_i+1}2-j],n_i])
\end{equation}
and
\begin{equation}\label{eq: transpose of unitary}
    \pi^t=U(\delta[\sigma_1,r_1],n_1)\times \cdots \times
    U(\delta[\sigma_m,r_m],n_m)\in \aaa_F^u.
\end{equation}
\end{lemma}
\begin{proof}
Let $a_i=a(\sigma_i,r_i,n_i).$ It follows from \eqref{eq: transpose
of U} that
\[
    \pi=\rep{a_1}^t \times \cdots \times \rep{a_m}^t=
    (\rep{a_1} \times \cdots \times \rep{a_m})^t
\]
and since $t$ is an involution on $\aaa_F$ that $\rep{a_1} \times
\cdots \times \rep{a_m} \in \aaa_F.$ Thus, it follows from
\cite[Proposition 8.4]{MR584084} that $\rep{a_1} \times \cdots
\times \rep{a_m}=\rep{a_1+\cdots+a_m}.$ In other words
$\pi=\rep{a_1+\cdots+a_m}^t$ and therefore by definition
\[
    \rec_F(\pi)=\tau(\rec_F^\circ(a_1+\cdots+a_m))=\oplus_{i=1}^m
    \tau(\rec_F^\circ(a_i)).
\]
The identity \eqref{eq: rec of unitary} now follows from the
definition of $\tau(\rec_F^\circ(a_i)).$ Note that \eqref{eq:
transpose of U} implies that
\[
    \pi^t=U(\delta[\sigma_1,r_1],n_1)\times \cdots \times
    U(\delta[\sigma_m,r_m],n_m)
\]
and the classification of Tadic therefore implies that $\pi^t \in
\aaa_F^u.$ Thus we get \eqref{eq: transpose of unitary}.
\end{proof}
Applying \eqref{eq: rec of unitary} to $\pi^t$ and comparing with
\eqref{eq: explicit SL(2)-type} Theorem \ref{thm: SL(2)-type}
follows from the definitions.
\end{proof}

From now on for every $\pi \in \aaa_F$ we denote by
\begin{equation}\label{eq: def of vv}
    \vv(\pi)=P_F(\rec_F(\pi^t))
\end{equation}
the $SL(2)$-type of $\pi.$

\section{Klyachko models}\label{sec: Klyachko}
For positive integers $r$ and $k$ denote by $U_r$ the subgroup of
upper triangular unipotent matrices in $GL_r(F)$ and by $Sp_{2k}(F)$
the symplectic group in $GL_{2k}(F).$ Fix a decomposition $n=r+2k.$
Let
\[
    H_{r,2k} =\{
\left(
\begin{array}{cc}
  u  & X   \\
  0  &  h  \\
\end{array}
\right): u \in U_r,\, X \in M_{r \times 2k}(F),\,h \in Sp_{2k}(F)\}.
\]
Let $\psi$ be a non trivial character of $F.$ For $u=(u_{i,j}) \in
U_r$ let
%\begin{equation}\label{eq: gen char}
\[
    \psi_r(u)=\psi(u_{1,2}+\cdots+u_{r-1, r})
\]
%\end{equation}
and let $\psi_{r,2k}$ be the character of $H_{r,2k}$ defined by
%\begin{equation}\label{eq: char on H}
\[
    \psi_{r,2k}
    \left(
        \begin{array}{cc}
        u  & X   \\
        0  &  h  \\
        \end{array}
    \right)=\psi_r(u).
\]
%\end{equation}
We refer to the space
\[
\mix_{r,2k}=\mbox{Ind}_{H_{r,2k}}^{GL_{n}(F)}(\psi_{r,2k})
\]
as a \emph{Klyachko model} for $GL_n(F).$ Here $\Ind$ denotes the
functor of non-compact smooth induction.

In \cite[Corollary 1]{OS3}
we showed that for any $\pi \in \aaa_F^u(n)$ there exists a unique
decomposition
\[
n=r(\pi)+2k(\pi)
\]
such that
\[
\Hom_{GL_n(F)}(\pi,\mix_{r(\pi),2k(\pi)})\ne 0
\]
and that in fact
$\dim_\cc(\Hom_{GL_n(F)}(\pi,\mix_{r(\pi),2k(\pi)}))=1.$
\begin{definition}
For $\pi \in \aaa_F^u,$ the \emph{Klyachko type} of $\pi$ is the
ordered pair $(r(\pi),2k(\pi)).$
\end{definition}
In fact, for $\aaa_F^u$ \cite[Theorem 8]{os} provides a receipt in order to
read the Klyachko type off Tadic's classification.
Based on \eqref{eq: explicit SL(2)-type}, our results can be reinterpreted
by the formula
\begin{equation}\label{eq: kly in terms of SL(2)-type}
    r(\pi)=\odd(\vv(\pi)),\ \pi \in \aaa_F^u.
\end{equation}

\section{Base change-The main results}\label{sec: statements}

Let $E$ be a finite extension of $F.$ Denote by
$\res_{E/F,n}:\weil_F(n)\to \weil_E(n)$ the map defined by
$\res_{E/F,n}(([\rho],f))=([\rho_{|W_E}],f).$ For $n\ge 1$ the
\emph{base change} $\bc_{E/F}(\pi)\in \aaa_E(n)$ of $\pi \in
\aaa_F(n)$ is defined by
\[
    \rec_E(\bc_{E/F}(\pi))=\res_{E/F}(\rec_F(\pi)).
\]
\begin{theorem}\label{thm: main}
Let $E/F$ be a finite extension of $p$-adic fields and let $\pi$ be
a smooth, irreducible and unitarizable representation of $GL_n(F).$
Then $\bc_{E/F}(\pi)$ is a smooth, irreducible and unitarizable
representation of $GL_n(E)$ and
\[
    \vv(\pi)=\vv(\bc_{E/F}(\pi)),
\]
i.e. $\pi$ and $\bc_{E/F}(\pi)$ have the same $SL(2)$-type.
\end{theorem}
As a consequence we have the following.
\begin{corollary}\label{cor: main}
Under the assumptions of Theorem \ref{thm: main} we have
\[
    r(\pi)=r(\bc_{E/F}(\pi)),
\]
i.e. $\pi$ and $\bc_{E/F}(\pi)$ have the same Klyachko type.
\end{corollary}
Corollary \ref{cor: main} is straightforward from Theorem \ref{thm:
main} and \eqref{eq: kly in terms of SL(2)-type}.
\section{Proof of the main result}\label{sec: proofs}
\begin{lemma}\label{lemma: bc for cuspidal}
Let $E/F$ be a finite extension.
For $\sigma \in \aaa_F^\circ\cap \aaa_F^u$ there exist $\sigma_1,\dots,\sigma_m \in
\aaa_E^\circ\cap \aaa_E^u$ such that
\[
    \bc_{E/F}(\sigma)=\sigma_1\times \cdots \times \sigma_m.
\]
\end{lemma}
\begin{proof}
Recall that a representation in $ \aaa_F^\circ$ is unitarizable if and only if
its central character is unitary. Let $\rho$ be the irreducible representation of $W_F$ such that $\rec_F(\sigma)=([\rho],\triv_n).$ It follows from
\eqref{eq: central char} that $\rho$ has a unitary central character and therefore it has a unitary structure. Thus, the restriction $\rho_{|W_E}$ to $W_E$ also has a unitary structure and therefore each of its irreducible componencts has a unitary central character. The lemma follows by applying
\eqref{eq: rec of unitary} to $\res_{E/F}(\rec_F(\sigma))$.
\end{proof}
\begin{proposition}\label{prop: bc of Zel inv}
Let $E/F$ be a finite extension and let $\pi\in \aaa_F^u$ then
$\bc(\pi)\in \aaa_E^u$ and
\begin{equation}\label{eq: bc and transpose commute}
    \bc_{E/F}(\pi^t)=\bc_{E/F}(\pi)^t.
\end{equation}
\end{proposition}
\begin{proof}
Let $\pi \in \aaa_F^u$ be of the form \eqref{eq: general unitary}.
By Lemma \ref{lemma: bc for cuspidal} there exist $\sigma_{i,k}\in
\aaa_E^\circ,\,i=1,\dots,m,\,k=1,\dots,t_i$ such that
\[
\bc_{E/F}(\sigma_i)=\sigma_{i,1}\times \cdots \times \sigma_{i,t_i}.
\]
Let $\rho_i=\rec_F^\circ(\sigma_i)$ and
$\rho_{i,k}=\rec_E^\circ(\sigma_{i,k}).$ Thus,
\[
\res_{E/F}(\rho_i)=\oplus_{k=1}^{t_i}\rho_{i,k}.
\]
It follows from \eqref{eq: rec of unitary} that
\begin{equation}\label{eq: res of rec}
    \res_{E/F}(\rec_F(\pi))=\oplus_{i=1}^m
    \oplus_{j=1}^{r_i}\oplus_{k=1}^{t_i}
    \tau(\Delta[\sigma_{i,k}[\frac{r_i+1}2-j],n_i]).
\end{equation}
On the other hand, let
\[
    \Pi=\times_{i=1}^m \times_{k=1}^{t_i}U(\delta[\sigma_{i,k},n_i],r_i)
\]
Since $\pi\in \aaa_F^u,$ the classification of Tadic implies that
$\Pi \in \aaa_E^u$ and by \eqref{eq: rec of unitary} applied to $E$
instead of $F$ we have
\begin{equation}\label{eq: rec_E id}
    \rec_E(\Pi)=\oplus_{i=1}^m
    \oplus_{j=1}^{r_i}\oplus_{k=1}^{t_i}
    \tau(\Delta[\sigma_{i,k}[\frac{r_i+1}2-j],n_i]).
\end{equation}
Comparing \eqref{eq: res of rec} with \eqref{eq: rec_E id} we obtain
that $\Pi=\bc_{E/F}(\pi)$ and in particular that $\bc_{E/F}(\pi)
\in\aaa_E^u.$ Applying this to
$\pi^t$ expressed by \eqref{eq: transpose of unitary} gives
\[
    \bc_{E/F}(\pi^t)=\times_{i=1}^m
    \times_{k=1}^{t_i}U(\delta[\sigma_{i,k},r_i],n_i).
\]
Applying \eqref{eq: transpose of unitary} now to $\bc_{E/F}(\pi)^t$
we obtain the identity \eqref{eq: bc and transpose commute}.
\end{proof}
It is straightforward from the definitions that
\begin{equation}\label{eq: trivial id}
    P_F(\rec_F(\pi))=P_E(\rec_E(\bc_{E/F}(\pi)),\ \pi \in \aaa_F.
\end{equation}
For $\pi \in \aaa_F^u,$ applying \eqref{eq: trivial id} to $\pi^t$
and then \eqref{eq: bc and transpose commute} we get that
\[
    P_F(\rec_F(\pi^t))=P_E(\rec_E(\bc_{E/F}(\pi)^t).
\]
The identity $\vv(\pi)=\vv(\bc_{E/F}(\pi))$ is now immediate from
\eqref{eq: def of vv}. This completes the proof of Theorem \ref{thm:
main}.

%\bibliographystyle{alpha}
%\def\bibdir{../Bibfiles/}
%\bibliography{\bibdir all.bib}
%\input{\bibdir shortbibline}

\end{document}